\documentclass[11pt,oneside]{article}
\usepackage{amssymb,amsmath,latexsym,amsfonts,amscd,amsthm,multirow,ctable,colortbl,mathdots}
\usepackage{fancyhdr}
\usepackage[headings]{fullpage}
\usepackage{stmaryrd}
\usepackage[all]{xy}
\usepackage{rotating}

\usepackage{setspace}

\usepackage{hyperref}

\newcommand{\bea}{\begin{eqnarray}} 
\newcommand{\eea}{\end{eqnarray}} 
\newcommand{\bee}{\begin{eqnarray*}} 
\newcommand{\eee}{\end{eqnarray*}} 
\newcommand{\al}{\begin{align*}} 
\newcommand{\eal}{\end{align*}} 
\newcommand{\be}{\begin{equation}} 
\newcommand{\ee}{\end{equation}} 
\newcommand{\eq}[1]{(\ref{#1})} 
\newcommand{\bem}{\begin{pmatrix}} 
\newcommand{\eem}{\end{pmatrix}}

\def\c{\gamma}

\def\e{\epsilon}

\def\inf{\infty}

\def\l{\lambda} 
\def\m{\mu}

\def\w{\omega}

\def\s{\sigma}            
\def\t{\tau} 
\def\th{\theta}

\def\O{\Omega}

\newcolumntype{R}{ >{$}r <{$}}
\newcolumntype{C}{ >{$}c <{$}}

\def\ll{\ell}

\newcommand{\mc}[1]{\mathcal{#1}}

\newcommand{\comment}[1]{}

\newcommand{\RR}{{\mathbb R}}
\newcommand{\CC}{{\mathbb C}}
\newcommand{\ZZ}{{\mathbb Z}}
\newcommand{\QQ}{{\mathbb Q}}
\newcommand{\HH}{{\mathbb H}}

\newcommand{\ii}{{\bf i}}

\newcommand{\tpi}{2\pi\ii}

\def\jac{\operatorname{j}}

\newcommand{\Id}{\operatorname{Id}}

\newcommand{\ex}{\operatorname{e}} 

\newcommand{\SL}{\operatorname{\textsl{SL}}}      

\newcommand{\G}{\Gamma}	
\newcommand{\g}{\gamma}	

\newtheorem{thm}{Theorem}[section]

\newtheorem{lem}[thm]{Lemma}

\theoremstyle{definition}

\theoremstyle{remark}

\numberwithin{equation}{section}

\begin{document}

\setstretch{1.4}

\title{
    \textsc{
    	On the Discrete Groups of Mathieu Moonshine
	}
    }

\author{
	Miranda C. N. Cheng\footnote{Institut de Math\'ematiques de Jussieu, UMR 7586,
	Universit\'e Paris 7, Paris, France.
	\newline\indent\indent
	{\em E-mail:} {\tt chengm@math.jussieu.fr}
	}\\
	John F. R. Duncan\footnote{
         Department of Mathematics,
         Case Western Reserve University,
         Cleveland, OH 44106,
         U.S.A.
         \newline\indent\indent
         {\em E-mail:} {\tt john.duncan@case.edu}
          }
        \vspace{23pt}\\
{\textit{Dedicated to Igor B. Frenkel on the occasion of his $60$th birthday.}}        
}

\date{}

\maketitle

\abstract{
We prove that a certain space of cusp forms for the Hecke congruence group of a given level is one-dimensional if and only if that level is the order of an element of the second largest Mathieu group. As such, our result furnishes a direct analogue of Ogg's observation that the normaliser of a Hecke congruence group of prime level has genus zero if and only if that prime divides the order of the Fischer--Griess monster group. The significance of the cusp forms under consideration is explained by the Rademacher sum construction of the McKay--Thompson series of Mathieu moonshine. Our result supports a conjectural characterisation of the discrete groups and multiplier systems arising in Mathieu moonshine.
}

\clearpage

\tableofcontents

\section{Introduction}\label{sec:intro}

Ogg established one of the most surprising results relating finite groups to modular functions---a harbinger of monstrous moonshine, cf. \cite{conway_norton}---by proving  \cite{Ogg_AutCrbMdl} that the normalizer $\G_0(p)^+$ of the Hecke congruence group 
\begin{gather}\label{eqn:intro:G0p}
	\G_0(p)=\left\{
	\begin{pmatrix}
	a&b\\c&d
	\end{pmatrix}
	\in\SL_2(\ZZ)
	\mid
	c\equiv 0\pmod{p}
	\right\}
\end{gather}
at prime level $p$ determines a genus zero quotient of the upper half-plane $\HH$ (cf. (\ref{eqn:sums:XG})) if and only if $p$ is a prime dividing the order of the (at that time conjectural) Fischer--Griess monster sporadic group. Ogg offered a bottle of Jack Daniels whiskey (cf. \cite[Rmq.1]{Ogg_AutCrbMdl}) to anybody who could explain this fact (and this is perhaps part of the reason the term moonshine has found the r\^ole that it has in mathematics).

In this article we establish an analogue of Ogg's result for Mathieu moonshine. 
Consider the weight 3/2 cusp form given by the unary theta series 
\begin{gather}\label{eqn:intro:utheta}
	\eta(\t)^3=\sum_{m\in\ZZ}(4m+1)q^{(4m+1)^2/8},
\end{gather}
which coincides with the third power of the Dedekind eta function (cf. (\ref{eqn:fun:Ded})) according to an identity due to Euler.  
We prove that a certain space of cusp forms for $\G_0(p)$ with $p$ prime is one-dimensional and spanned by the above unary theta series if and only if $p$ is a prime dividing the order of the largest Mathieu sporadic group, $M_{24}$. In fact we prove more than this: the positive integers $n$ for which $\G_0(n)$ has the above cusp form property are exactly those that arise as orders of elements of $M_{24}$ having a fixed point in the {\em defining} permutation representation (i.e. the unique non-trivial permutation representation on $24$ points); equivalently, these values of $n$ are just those that arise as orders of elements of the second largest Mathieu sporadic group, $M_{23}$. The significance of the cusp form (\ref{eqn:intro:utheta}) and the specific space of cusp forms to be considered will soon become apparent.

Mathieu moonshine was initiated by the observation \cite{Eguchi2010} of Eguchi--Ooguri--Tachikawa that the low order (and non-polar) coefficients of a certain holomorphic function 
\begin{gather}
	H(\t)=-2q^{-1/8}+90q^{7/8}+462q^{15/8}+1540q^{23/8}+4554q^{31/8}+11592q^{39/8}+\ldots
\end{gather}
on the upper half-plane (we set $q=e^{2\pi\ii\t}$ throughout) are simple positive-integer combinations of degrees of irreducible representations of $M_{24}$. For example, $90=45+45$, $462=231+231$, and $1540=770+770$, \&c. (cf. \cite{ATLAS}). In \cite{Eguchi2010} (see also \cite{Eguchi2009a}) the function $H(\t)$ was obtained by decomposing the elliptic genus of a $K3$ surface into characters of the (small) $N=4$ superconformal algebra; it is essentially (up to subtraction of the polar term and multiplication by $q^{1/8}$) the generating function for multiplicities of the characters of massive (non-BPS) representations. 

More significantly for the purpose of this paper, the function $H(\t)$ is 
a weak mock modular form of weight $1/2$ for the modular group $\SL_2(\ZZ)$ with shadow equal to $24\,\eta^3$, meaning that
\begin{gather}
	H(\g\t)\e(\g)^{-3}\jac(\g,\t)^{1/2}
		+\frac{e^{\pi\ii/4}}{2}
	\int_{-\g^{-1}\infty}^{\infty}
	\frac{24\,\eta(z)^3}{\sqrt{z+\t}}{\rm d}z=H(\t)
\end{gather}
for $\t\in\HH$ and $\g\in\SL_2(\ZZ)$ and $\jac(\g,\t)=(c\t+d)^{-1}$ in case $\g=\left(\begin{smallmatrix}a&b\\c&d\end{smallmatrix}\right)$. In fact $H(\t)$ is uniquely determined amongst such objects (see \cite[\S8]{Cheng2011}) by having polar part $-2q^{-1/8}$. 

The observation relating coefficients of $H$ to degrees of irreducible representations of $M_{24}$ suggests the existence of a graded $M_{24}$-module $K=\bigoplus_{n>0}K_{n-1/8}$ with the property that 
\begin{gather}\label{eqn:intro:HK}
	H(\t)=-2q^{-1/8}+\sum_{n>0}(\dim K_{n-1/8})q^{n-1/8}.
\end{gather}
If such a module exists then one can expect to obtain interesting functions $H_g(\t)$---{\em McKay--Thompson series} for $M_{24}$---by replacing $\dim K_{n-1/8}$ with ${\rm tr}_{K_{n-1/8}}(g)$ in (\ref{eqn:intro:HK}) for $g\in M_{24}$. 
\be\label{eqn:intro:HKg}
H_g(\t) =
	-2q^{-{1}/{8}}  + \sum_{n>0}{\rm tr}_{K_{n-1/8}}(g)q^{n-1/8}
\ee
Candidate expressions for the $H_g$ were determined in a series of papers, starting with \cite{Cheng2010_1} and the independent work \cite{Gaberdiel2010}, and concluding with \cite{Gaberdiel2010a} and \cite{Eguchi2010a}. A proof of the existence of the $M_{24}$-module $K$, with McKay--Thompson series given exactly as predicted in \cite{Cheng2010_1,Gaberdiel2010,Gaberdiel2010a,Eguchi2010a}, has appeared very recently \cite{Gannon_proof}; as yet no concrete construction of $K$ is known.

It develops that the McKay--Thompson series all enjoy good modular properties.  
To describe them we set $n_g$ to be the order of $g$, we write $\chi(g)$ for the number of fixed points of $g$ in the defining permutation representation and define $h_g$ to be the length of a minimal length cycle in a disjoint cycle decomposition of (the permutation induced by) $g$ (cf. \cite{Cheng2011,CheDun_M24MckAutFrms}). For instance, for an element $g$ in the (unique) conjugacy class with cycle shape $1^82^8$ we have $n_g=2$, $h_g=1$ and $\chi(g)=8$;  for an element $g$ in the (unique) conjugacy class with cycle shape $2^{12}$ we have $n_g=h_g=2$ and $\chi(g)=0$.
In particular,  $h_g=1$ and $\chi(g)\neq 0$ just when $g$ has a fixed point. By definition $h_g$ divides $n_g$ and by inspection $h_g$ also divides $12$ for each $g\in M_{24}$.  In this way we attach to each element $g\in M_{24}$ a discrete group $\G_g=\G_0(n_g)<\SL_2(\ZZ)$ and a multiplier system $\psi_g=\rho_{n|h}\e^{-3}$ with weight $1/2$, where $n=n_g$ and $h=h_g$. Here $\e$ denotes the multiplier system of the Dedekind eta function (cf. (\ref{eqn:fun:eps})). The function $\rho_{n|h}:\G_0(n)\to \CC^{\times}$ is defined for any pair of positive integers $n$ and $h$ such that $h$ divides $\gcd(n,12)$ by setting
\begin{gather}\label{eqn:rho}
	\rho_{n|h}
	\left(
	\begin{matrix}a&b\\c&d\end{matrix}
	\right)=\exp\left(-2\pi\ii\frac{cd}{nh}\right).
\end{gather}
Since $h$ is a divisor of $24$ we have $x\equiv y\pmod{h}$ whenever $xy\equiv 1\pmod{h}$ (cf. \cite[\S3]{conway_norton}) and it follows from this that $\rho_{n|h}$ is actually a morphism of groups; the kernel is evidently $\G_0(nh)$.
Now we may describe the modularity of all the McKay--Thompson series $H_g$ by stating that for arbitrary $g\in M_{24}$ the function $H_g$ is a weak mock modular form for $\G_0(n_g)$ with weight $1/2$ and multiplier $\rho_{n|h}\e^{-3}$, where $n=n_g$ and $h=h_g$, and the shadow of $H_g$ is $\chi(g)\eta^3$. See (\ref{eqn:sums:gtwact}) for the definition of weak mock modular form. Notice that the functions $H_g$ for $g\in M_{24}$ with and without fixed points have qualitatively different modular behaviour: the former are weak mock modular forms on $\G_0(n_g)$ with non-trivial shadow while the latter are actually weak modular forms. 

In \cite{Cheng2011} it is shown that a uniform construction of the McKay--Thompson series $H_g$ may be obtained using Rademacher sums. 
The analogous statement for monstrous moonshine was established in \cite{DunFre_RSMG}, while in \cite{UM} it was conjectured that an analogous Rademacher sum construction also exists for the McKay--Thompson series of umbral moonshine. The fact that the McKay--Thompson series in these various settings may all be  constructed using Rademacher sums indicates a deep relation between Rademacher sums and moonshine. 

To explore this relation further, and to motivate the spaces of cusp forms we study in this work, recall (from \cite{Cheng:2012qc}, for example) that given a group $\G<\SL_2(\RR)$ commensurable with $\SL_2(\ZZ)$, a multiplier system $\psi$ for $\G$, a compatible weight $w$ and a compatible index $\mu$, we may consider the Rademacher sum $R_{\G,\psi,w}^{[\m]}$ (we refer to \cite{Cheng:2012qc} for the definition). In the case that it converges $R^{[\m]}_{\G,\psi,w}$ defines a weak mock modular form for the group $\G$ with multiplier $\psi$ and weight $w$, and Fourier expansion of the form $R_{\G,\psi,w}^{[\m]}(\t)=q^{\mu}+O(1)$. (We typically have $w<1$ and $\m<0$ in applications.) The shadow of $R_{\G,\psi,w}^{[\m]}$ is a modular form for $\G$---a cusp form in case $\mu< 0$---with the inverse multiplier $\overline{\psi}$ and dual weight $2-w$, and is itself a Rademacher sum. We refer to \cite{Cheng:2012qc} for an exposition. 

Write $S_{\psi,w}(\G)$ for the space of cusp forms for $\G$ with multiplier $\psi$ and weight $w$, so that $S_{\psi,w}(\G)$ is the space of possibilities for the shadow of a Rademacher sum $R^{[\m]}_{\G,\overline\psi,2-w}$. As we shall see presently the Rademacher sums of monstrous and Mathieu moonshine have the common feature that their corresponding spaces $S_{\psi,w}(\G)$ are extremely small; viz., zero or one dimensional, and this feature plays an important r\^ole in determining the Rademacher sums themselves. From this point of view Ogg's result and the result of the present paper may be viewed as two manifestations of a single (as yet empirical) principle of the Rademacher sum construction. 

In the case of monstrous moonshine the relevant Rademacher sums have trivial multiplier, weight $0$ and $\m=-1$ \cite{DunFre_RSMG} and thus their shadows lie in the spaces $S_2(\G)=S_{1,2}(\G)$ of cusp forms with trivial multiplier of weight $2$. 
The fact that the dimension of $S_2(\G)$ coincides with the genus of the compact Riemann surface $X_{\G}$ determined by $\G$ (cf. (\ref{eqn:sums:XG})) indicates the close connection between Rademacher sums and the genus zero property: the weight $0$ Rademacher sum $R^{[-1]}_{\G,1,0}$ is forced to have trivial shadow, and thus be $\G$-invariant (i.e. a weight $0$ weak modular form for $\G$), whenever $X_{\G}$ has genus zero since in that case the space of possible shadows $S_2(\G)$ is zero-dimensional. In \cite{DunFre_RSMG} it is shown that the converse is also true, so that a discrete group $\G$ determines a genus zero surface $X_{\G}$ if and only if the Rademacher sum $R^{[-1]}_{\G,1,0}$ is $\G$-invariant, and the genus zero property of monstrous moonshine is thus reformulated in terms of Rademacher sums. (As such, the main result of \cite{Cheng2011} may be regarded as verifing a natural analogue of the genus zero property for Mathieu moonshine;  we refer to \cite{CheDun_M24MckAutFrms,Cheng:2012qc} for a fuller discussion.) In particular, the equality of a monstrous McKay--Thompson series $T_g$ with $R^{[-1]}_{\G_g,1,0}$ (up to an additive constant, cf. \cite{DunFre_RSMG,Cheng:2012qc}) implies that $\G_g$ has genus zero. 
 
To recover the McKay--Thompson series of Mathieu moonshine corresponding to elements with fixed points we choose a positive integer $n$ and 
set $\G=\G_0(n)$ and $\psi=\e^{-3}$ (cf. (\ref{eqn:fun:eps})) and $w=1/2$ and also $\m=-1/8$. Then the space of possible shadows for the Rademacher sum $R^{[-1/8]}_{\G_0(n),\e^{-3},1/2}$ is $S_{\e^3,3/2}(\G_0(n))$, and we abbreviate this to $S_{\e^3,3/2}(n)$ to ease notation. We can see immediately that $S_{\e^{3},3/2}(n)$ has positive dimension for all positive $n$ since the cusp form $\eta^3$ (cf. (\ref{eqn:intro:utheta})) is a non-zero element of $S_{\e^3,3/2}(1)$, and $S_{\e^3,3/2}(m)$ embeds in $S_{\e^3,3/2}(n)$ whenever $m|n$. In the course of proving \cite{Cheng2011} that the McKay--Thompson series $H_g$ determined in the aforementioned articles \cite{Cheng2010_1,Gaberdiel2010,Gaberdiel2010a,Eguchi2010a} satisfy $H_g=-2R^{[-1/8]}_{\G_0(n),\e^{-3},1/2}$ for $n=n_g$ when $g$ has a fixed point, it was shown that $S_{\e^3,3/2}(n)$ is one-dimensional, spanned by $\eta^3$, whenever $n$ is the order of an element of $M_{23}$ (cf. \cite[\S8.1]{Cheng2011}). In particular, the restricted nature of the space $S_{\e^3,3/2}(n)$ was crucial for the purposes of determining the Rademacher sums associated to elements of the subgroup $M_{23}$. 

In this article we make an important step towards a characterisation of the discrete groups and multipliers of Mathieu moonshine, and in so doing strengthen the analogy with the monstrous case just described, by proving the converse statement. We identify an isomorphism between $S_{\e^3,3/2}(n)$ and a certain space of meromorphic functions on the modular curve $X_0(n)$ (cf. (\ref{eqn:sums:XG}-\ref{eqn:pre:G0n})) and use this together with the Riemann--Roch theorem to prove that $\dim S_{\e^3,3/2}(n)=1$ if and only if $n$ is the order of an element of $M_{23}$. Since there is an element in $M_{23}$ with order $p$ for any prime $p$ dividing the order of $M_{24}$ we obtain a direct analogue of Ogg's result characterising the primes $p$ for which $\G_0(p)^+$ has genus zero in terms of the Monster group: the primes $p$ for which $\G_0(p)$ has a unique (up to scale) cusp form of weight $3/2$ with multiplier $\e^3$ are exactly those that divide the order of $M_{24}$, which are also just the primes $p$ such that $p+1$ divides $24$.

The rest of the paper is organised as follows. In \S\ref{sec:pre} we collect preliminary definitions and notations. In \S\ref{sec:fms} we present the proof of the main theorem of the paper. We conclude in \S\ref{sec:persp} with some perspectives on possible future developments, including a conjectural characterisation of the discrete groups and multipliers of Mathieu moonshine.

\section{Preliminaries}\label{sec:pre}

The group $\SL_2(\RR)$ acts naturally on the upper half-plane $\HH$ by orientation preserving isometries according to the rule 
\begin{gather}\label{eqn:pre:sl2actn}
	\bem
	a&b\\c&d
	\eem
	\t=\frac{a\t+b}{c\t+d}.
\end{gather}
For $\g\in \SL_2(\RR)$ with lower row $(c,d)$ define 
\begin{gather}
	\jac(\g,\t)=(c\t+d)^{-1}
\end{gather}
so that $\jac(\g,\t)^2$ is the derivative (with respect to $\t$) of the action (\ref{eqn:pre:sl2actn}) when $\g=\left(\begin{smallmatrix}a&b\\c&d\end{smallmatrix}\right)$. 
For $\G$ a finite index subgroup of the modular group $\SL_2(\ZZ)$ and for $w\in \RR$ say that a function $\psi:\G\to \CC$ is a {\em multiplier system} for $\G$ with weight $w$ if
\begin{gather}\label{eqn:sums:mult}
	\psi(\g_1)\psi(\g_2)\jac(\g_1,\g_2\t)^{w}\jac(\g_2,\t)^{w}
	=
	\psi(\g_1\g_2)\jac(\g_1\g_2,\t)^{w}
\end{gather}
for all $\g_1,\g_2\in \G$ where here and everywhere else in this paper we choose the principal branch of the logarithm (cf. (\ref{eqn:fun:pbranch})) in order to define the exponential $x\mapsto x^s$ in case $s$ is not an integer. 
Given a multiplier system $\psi$ for $\G$ with weight $w$ we may define the {\em $(\psi,w)$-action} of $\G$ on the space $\mc{O}(\HH)$ of holomorphic functions on the upper half-plane by setting
\begin{gather}\label{eqn:sums:psiw_actn}
	(f|_{\psi,w}\g)(\t)=f(\g\t)\psi(\g)\jac(\g,\t)^{w}
\end{gather}
for $f\in \mc{O}(\HH)$ and $\g\in \G$. We then say that $f\in \mc{O}(\HH)$ is an {\em unrestricted modular form} with multiplier $\psi$ and weight $w$ for $\G$ in the case that $f$ is invariant for this action; i.e. $f|_{\psi,w}\g=f$ for all $\g\in\G$. Since $(-\g)\t=\g\t$ and $\jac(-\Id,\t)^{w}=e^{-\pi\ii w}$ (cf. (\ref{eqn:fun:pbranch})) the multiplier $\psi$ must satisfy the {\em consistency condition} $\psi(-\Id)=e^{\pi\ii w}$ in order that the corresponding space(s) of unrestricted modular forms be non-vanishing when $-
\Id\in \G$. 

We assume throughout that the multiplier $\psi$ for $\G$ is of the form $\psi=\rho\tilde{\psi}$ where $\rho:\G\to\CC^{\times}$ is a morphism of groups and $\tilde{\psi}$ is a multiplier for $\SL_2(\ZZ)$. With this understanding we say that an unrestricted modular form $f$ for $\G$ with multiplier $\psi$ and weight $w$ is a {\em weak modular form} in case $f$ has at most exponential growth at the cusps of $\G$; i.e. in case there exists $C>0$ such that $(f|_{\tilde{\psi},w}\s)(\t)=O(e^{C\Im(\t)})$ as $\Im(\t)\to \inf$ for any $\s\in\SL_2(\ZZ)$. We say that $f$ is a {\em modular form} if $(f|_{\tilde{\psi},w}\s)(\t)$ remains bounded as $\Im(\t)\to \inf$ for any $\s\in\SL_2(\ZZ)$, and we say $f$ is a {\em cusp form} if $(f|_{\tilde{\psi},w}\s)(\t)\to 0$ as $\Im(\t)\to\inf$ for any $\s\in\SL_2(\ZZ)$.

Suppose that $\psi$ is a multiplier system for $\G$ with weight $w$ and $g$ is a modular form for $\G$ with the {inverse multiplier system} $\bar{\psi}:\g\mapsto \overline{\psi(\g)}$ and {\em dual} weight $2-w$. Then we may use $g$ to twist the $(\psi,w)$-action of $\G$ on $\mc{O}(\HH)$ by setting
\begin{gather}\label{eqn:sums:gtwact}
	\left(f|_{\psi,w,g}\g\right)(\t)
	=
	f(\g\t)\psi(\g)\jac(\g,\t)^{w}
	+\left(\frac{\ii}{4}\right)^{1-w}
	\int_{-\g^{-1}\infty}^{\infty}(z+\t)^{-w}\overline{g(-\bar{z})}{\rm d}z.
\end{gather}
A {\em weak mock modular form} for $\G$ with multiplier $\psi$, weight $w$, and {\em shadow} $g$ is a holomorphic function $f$ on $\HH$ that is invariant for the $(\psi,w,g)$-action of $\G$ defined in (\ref{eqn:sums:gtwact}) and which has at most exponential growth at the cusps of $\G$ (i.e. there exists $C>0$ such that $(f|_{\tilde{\psi},w}\s)=O(e^{C\Im(\t)})$ for all $\s\in\SL_2(\ZZ)$ as $\Im(\t)\to\inf$ where $\tilde{\psi}$ is as in the previous paragraph.) A weak mock modular form is called a {\em mock modular form} in case it is bounded at every cusp. From this point of view a (weak) modular form is a (weak) mock modular form with vanishing shadow. The notion of mock modular form developed from Zwegers' ground breaking work \cite{zwegers} on Ramanujan's mock theta functions. We refer to \cite{zagier_mock} for an excellent review. The notion of mock modular form is closely related to the notion of {\em automorphic integral} which was introduced by Niebur in \cite{Nie_ConstAutInts}; see \cite{Cheng:2012qc} for a discussion of this. 

Since $\G$ is assumed to be a subgroup of $\SL_2(\ZZ)$ of finite index its natural action on the boundary $\hat{\RR}=\RR\cup\{\inf\}$ of $\HH$ restricts to $\hat{\QQ}=\QQ\cup\{\inf\}$. The orbits of $\G$ on $\hat{\QQ}$ are called the {\em cusps} of $\G$. The quotient space
\begin{gather}\label{eqn:sums:XG}
	X_{\G}=\G\backslash\HH\cup\hat{\QQ}
\end{gather}
is naturally a compact Riemann surface (cf. e.g. \cite[\S1.5]{Shi_IntThyAutFns}). We say that $\G$ has {\em genus} $g$ in case $X_{\G}$ has genus $g$ as an orientable surface.

In this paper we will be concerned primarily with the case that $\G$ is the {\em Hecke congruence group} of level $n$, denoted $\G_0(n)$, for some integer $n$.
\begin{gather}\label{eqn:pre:G0n}
	\G_0(n)=\left\{
	\begin{pmatrix}
	a&b\\c&d
	\end{pmatrix}
	\in\SL_2(\ZZ)
	\mid
	c\equiv 0\pmod{n}
	\right\}
\end{gather}
We write $X_{0}(n)$ for $X_{\G}$ when $\G=\G_0(n)$.

\section{Cusp forms}\label{sec:fms}

Recall that $S_{\e^3,3/2}(n)$ denotes the space of cusp forms for $\G_0(n)$ with multiplier $\e^3$ and weight $3/2$. Then the function $\eta^3$ belongs to $S_{\e^3,3/2}(n)$ for every positive integer $n$, so in particular $S_{\e^3,3/2}(n)$ has dimension at least $1$ for all $n$. We will prove that the values of $n$ for which $\dim S_{\e^3,3/2}(n)=1$ are exactly those that arise as the order of an element of the sporadic group $M_{23}$; viz., $n\in\{1,2,3,4,5,6,7,8,11,14,15,23\}$.

Observe that if $g\in S_{\e^3,3/2}(n)$ then $\tilde{g}=g\eta^{-3}$ is a meromorphic function on $X=X_0(n)$ with poles only at the cusps of $\G=\G_0(n)$, and the order of the pole at a cusp $x\in\G\backslash\hat{\QQ}$ say is bounded from above by the order of vanishing of $\eta^3$ at $x$. For a more precise statement define a divisor $D$ on $X$ by setting 
\begin{gather}\label{eqn:fms:D}
	D=\sum_{x\in\G\backslash\hat{\QQ}}
	\left(\left\lceil \frac{w_x}{8}\right\rceil-1\right)x
\end{gather}
where $w_x$ denotes the width of $\G$ at the cusp $x$. Write $\mc{K}_X(D)$ for the vector space composed of meromorphic functions $f$ on $X$ satisfying $(f)+D\geq 0$. Then $g\eta^{-3}$ belongs to $\mc{K}_X(D)$ whenever $g\in S_{\e^3,3/2}(n)$, and conversely, if $\tilde{g}\in\mc{K}_X(D)$ then $\tilde{g}\eta^3$ belongs to $S_{\e^3,3/2}(n)$. So multiplication by $\eta^3$ defines an isomorphism of vector spaces $\mc{K}_X(D)\to S_{\e^3,3/2}(n)$. According to the Riemann--Roch Theorem we have 
$\dim\mc{K}_X(D)-\dim\O_X(D)={\rm deg}(D)+1-{\rm genus}(X)$
where $\O_X(D)$ denotes the space of holomorphic differentials $\w$ on $X$ satisfying $(\w)-D\geq 0$. In particular then 
\begin{gather}\label{eqn:fms:RRbound}
S_{\e^3,3/2}(n)
\geq \deg(D)+1-{\rm genus}(X)
\end{gather}
when $D$ is given by (\ref{eqn:fms:D}) and $X=X_0(n)$.
In the case that $X=X_0(n)$ we have the explicit formula (cf. \cite[\S6.1]{Ste_MdlrFrmsCmpApp})
\begin{gather}\label{eqn:fms:genus}
	{\rm genus}(X_0(n))=1+\frac{1}{12}i(n)-\frac{1}{4}\mu_2(n)-\frac{1}{3}\mu_3(n)-\frac{1}{2}c(n)
\end{gather}
where $i(n)$ is the index of $\G_0(n)$ in the modular group $\SL_2(\mathbb Z)$, and $c(n)=\#\G_0(n)\backslash\hat{\QQ}$ is the number of cusps of $\G_0(n)$, and 
$\mu_2(n)$ and $\mu_3(n)$ are defined by setting
\begin{gather}
	\mu_2(n)=\begin{cases}
	0&\text{if $4|n$,}\\
	\prod_{p|n}\left(1+\left(\frac{-4}{p}\right)\right)&\text{otherwise},
	\end{cases}\\
	\mu_3(n)=\begin{cases}
	0&\text{if $2|n$ or $9|n$,}\\
	\prod_{p|n}\left(1+\left(\frac{-3}{p}\right)\right)&\text{otherwise},
	\end{cases}
\end{gather}
where $\left(\frac{k}{p}\right)$ denotes the Kronecker symbol (cf. \S\ref{sec:fun}). Substituting (\ref{eqn:fms:genus}) and the expression (\ref{eqn:fms:D}) for $D$ into (\ref{eqn:fms:RRbound}) we thus obtain the {lower bound}
\begin{gather}\label{eqn:fms:strongbound}
\dim S_{\e^3,3/2}(n)\geq \sum_{x\in\G\backslash\hat{\QQ}}
	\left\lceil \frac{w_x}{8}\right\rceil
	-\frac{1}{12}i(n)-\frac{1}{2}c(n)+\frac{1}{4}\mu_2(n)+\frac{1}{3}\mu_3(n)
\end{gather}
on $\dim S_{\e^3,3/2}(n)$. A weaker but still useful lower bound is the formula
\begin{gather}\label{eqn:fms:weakbound}
\dim S_{\e^3,3/2}(n)\geq \sum_{x\in\G\backslash\hat{\QQ}}
	\left\lceil \frac{w_x}{8}\right\rceil
	-\frac{1}{12}i(n)-\frac{1}{2}c(n),
\end{gather}
obtained by ignoring the contribution to ${\rm genus}(X_0(n))$ of elliptic points on $X_0(n)$. Since $i(n)$ coincides with the sum of the widths of the cusps of $\G_0(n)$ (cf. e.g. \cite[\S5.3]{DunFre_RSMG}) we have the crude lower bound
\begin{gather}\label{eqn:fms:Scrude}
	\dim S_{\e^3,3/2}(n)\geq \frac{1}{24}i(n)-\frac{1}{2}c(n)
\end{gather}
which shows at a glance that $S_{\e^3,3/2}(n)$ has dimension greater than $1$ for sufficiently large $n$ since the index of $\G_0(n)$ in the modular group grows faster with $n$ than does the number of cusps of $\G_0(n)$. Precise formulas for $i(n)$ and $c(n)$ are as follows (cf. e.g. \cite[\S6.1]{Ste_MdlrFrmsCmpApp}).
\begin{align}
	{i}(n)&= \prod_{p|n}(p^{\nu_p(n)}+p^{\nu_p(n)-1})\label{eqn:fms:ind}\\
	c(n) &=\sum_{d|n}\phi(\gcd(d,n/d))\label{eqn:fms:cuspnum}
\end{align}
In (\ref{eqn:fms:ind}) we write $\nu_p(n)$ for the greatest positive integer such that $p^{\nu_p(n)}$ divides $n$. In (\ref{eqn:fms:cuspnum}) we write $\phi$ for the Euler totient function.

The cusps of $\G_0(n)$ are indexed by equivalence classes of pairs $(a,d)$ where $d|n$ and $\gcd(a,d)=1$, and $(a,d)\sim (a',d')$ just when $d=d'$ and $a\equiv a'$ modulo $\gcd(d,n/d)$. Writing $d= \prod_{p|n} p^{\nu_p(d)}$, the width of the cusp corresponding to a pair $(a,d)$ is given by 
\be
\prod_{p|n} p^{(\nu_p(n)-2\nu_p(d))H[\nu_p(n)-2\nu_p(d)]}
\ee
where $H[n]$ is the Heaviside step function given by $0$ when $n<0$ and $1$ when $n\geq0$.

Using \eq{eqn:fms:strongbound} and \eq{eqn:fms:weakbound} together with (\ref{eqn:fms:ind}) and (\ref{eqn:fms:cuspnum}) we will now show that $\dim S_{\e^3,3/2}(n)>1$ for all $n$ except those that arise as the order of an element of the Mathieu group $M_{23}$. We start with the following result on prime powers.

\begin{lem}\label{prime_power_lemma}
If $n$ is a prime power then $\dim S_{\e^3,3/2}(n)>1$ for $n\notin\{1,2,3,5,7,11,23,4,8\}$.
 \end{lem}

\begin{proof}
First consider the case that $n=p$ is prime. We have $i(p)=p+1$ and $c(p)=2$. The cusp at infinity has width $1$ and the cusp represented by $0\in\hat{\QQ}$ has width $p$ so $\dim S_{\e^3,3/2}(p)\geq (p-2)/24$ according to (\ref{eqn:fms:weakbound}) and thus $\dim_{\e^3,3/2}(n)>1$ when $n$ is a prime greater than $23$. For the case that $n$ is prime it remains to show that $\dim S_{\e^3,3/2}(p)>1$ for $p\in \{13,17,19\}$ and this follows from (\ref{eqn:fms:strongbound}) together with the explicit computations $\mu_2(13)=\mu_2(17)=\mu_3(13)=\mu_3(19)=2$ and $\mu_2(19)=\mu_3(17)=0$.

Now suppose that $n=p^{2\l}$ for some $\l\geq 1$. Then we have $i(p^{2\l}) =p^{2\l}+p^{2\l-1} $ and $c(p^{2\l}) =p^{\l}+p^{\l-1} $, and there are $p^\l$ cusps of width 1, one cusp of width $p^{2\l}$, and $p^{\l-k}-p^{\l-k-1}$ cusps of width $p^{2k}$ for every $k=1,\dots,\l-1$. To get a simple bound on $\dim S_{\e^3,3/2}(p^{2\l})$ we use
the fact that $p^\l$ cusps have width $1$ and obtain the lower bound 
\begin{gather}
	\sum_{x\in \G\backslash\hat{\QQ}}\left\lceil \frac{w_x}{8}\right\rceil\geq \frac{p^{2\l}+p^{2\l-1}-p^{\l}}{8}+p^{\l},
\end{gather}
and then an application of (\ref{eqn:fms:weakbound}) yields
\begin{align}\label{even_power_prime}
\dim S_{\e^3,3/2}(p^{2\l}) &\geq \frac{p^{2\l}+p^{2\l-1}+9p^\l-12 p^{\l-1}}{24}.
\end{align}
Taking $\l=1$ in (\ref{even_power_prime}) we see that $\dim S_{\e^3,3/2}(n)>1$ whenever $n$ is the square of a prime greater than $2$, and taking $\l=2$ in (\ref{even_power_prime}) we see that $\dim S_{\e^3,3/2}(n)>1$ whenever $n$ is the fourth power of any prime. Using now the elementary fact that $S_{\e^3,3/2}(m)$ is naturally identified with a subspace of $S_{\e^3,3/2}(n)$ whenever $m$ is a divisor of $n$ we conclude that $\dim S_{\e^3,3/2}\geq 2$ whenever $n$ is a prime power not in the set $\{1,2,3,5,7,11,23,4,8\}$, as we required to show.
%
\end{proof}

\begin{lem}
If $n=\prod_{i=1}^k p_i$ is a square-free product of distinct primes, then $\dim S_{\e^3,3/2}(n)>1$ unless  $n\in\{6,14,15\}$.
\end{lem}

\begin{proof}
For $n=\prod_{i=1}^k p_i$ we have $i(n) = \prod_{i=1}^k (p_i+1)$ and $c(n) = 2^k$. Hence the crude estimate (\ref{eqn:fms:Scrude}) gives
\be\label{prime_prod}
\dim S_{\e^3,3/2}(n) > \frac{1}{24} \prod_{i=1}^k (p_i+1) - 2^{k-1}. 
\ee
Take the case $k=2$ and $n=pq$. Then the above formula shows that $\dim S_{\e^3,3/2}(n) \geq 2$ if both $p$ and $q$ are larger than 8. 
Without loss of generality let $p<q$. To refine (\ref{prime_prod}) we put  $\lceil \frac{p}{8} \rceil =1$ in (\ref{eqn:fms:weakbound}) and obtain
\be
\dim S_{\e^3,3/2}(pq) > \frac{pq}{8}+ \frac{q}{8}-\frac{1}{12} (p+1)(q+1) = \frac{1}{24} (p+1)(q-2)
\ee 
and thus ${\dim}S_{\e^3,3/2}(pq)>1$ for $pq$ not in the set $\{6,10,14,15,21\}$. 

For $k=3$ the crude estimate (\ref{prime_prod}) excludes all possibilities except for $n=30$ and $n=42$. Similarly, it excludes all possibilities with $k>3$, and so it remains to show that ${\rm dim}S_{\e^3,3/2}(n)>1$ for $n\in\{10,21,30,42\}$. This follows in each case from an explicit computation 
of $\mu_2(n)$ and $\mu_3(n)$ and an application of (\ref{eqn:fms:strongbound}).
\end{proof}

\begin{lem}
We have  $\dim S_{\e^3,3/2}(n)>1$ whenever $n>1$ is not a prime power and is not square-free.
 \end{lem}

\begin{proof}
From the fact, mentioned above, that $\dim S_{\e^3,3/2}(n)\geq S_{\e^3,3/2}(m)$ whenever $m$ divides $n$ we conclude from Lemma \ref{prime_power_lemma} that  the statement is automatically true unless $n=4p$ or $n=8p$ with $p\in\{3,5,7,11,23\}$.

The group $\G_0(4p)$ has two cusps of width $1$, two cusps of width $p$, one cusp of width $4$ and one with width $4p$. 
This gives deg$(D) > \frac{3p}{4}-3$.  On the other hand, we have $i(4p)=6p+6$ and $c(4p)=6$ and this gives ${\rm genus}(X) \leq \frac{p-3}{2}$, and so we obtain $\dim S_{\e^3,3/2}(4p)\geq \frac{p-2}{4}$ from (\ref{eqn:fms:RRbound}),
and hence $\dim S_{\e^3,3/2}(4p)>1$ unless $p=3$ or $p=5$. For $n=12$ we compute $\deg D=1$ and ${\rm genus}(X)=0$, and this gives $\dim S_{\e^3,3/2}(12)>1$ via (\ref{eqn:fms:RRbound}). Similarly for $n=20$ we have $\deg D=2$ and ${\rm genus}(X)=1$ and so $\dim S_{\e^3,3/2}(20)>1$ also. Therefore $\dim S_{\e^3,3/2}(4p)>1$ for all $p>2$.  Then it follows that $\dim S_{\e^3,3/2}(8p)\geq \dim S_{\e^3,3/2}(4p)>1$ for all $p>2$
and this finishes the proof. 
\end{proof}

For the sake of completeness we conclude with a proof of the result, established earlier in \cite{Cheng2011}, that $\dim S_{\e^3,3/2}(n)=1$ whenever $n$ is the order of an element of $M_{23}$.

\begin{lem}
If $n\in\{1,2,3,4,5,6,7,8,11,14,15,23\}$ then $\dim S_{\e^3,3/2}(n)=1$.
\end{lem}

\begin{proof} 

For $n\leq 8$ we have $D=0$. Since every holomorphic function on a compact Riemann surface is a constant we have $\mc{K}_X(D)=\mathbb C$ and as a result $\dim S_{\e^3,3/2}(n)=\dim\mc{K}_X(D)=1$. For $n\in\{11,14,15\}$ there is just one cusp of $\G_0(n)$ (the one represented by $1$) that has width larger than $8$. 
In these cases, denote by $x$ the image of $1\in\hat{\QQ}$ under the natural map $\HH\cup\hat{\QQ} \to X_0(n)$, then any $f\in \mc{K}_X(D)$ is either a constant or has a simple pole at $x$ and no other poles, but the latter is impossible since such an $f$ would induce an isomorphism between the Riemann sphere and $X_0(n)$, which is a genus one curve for $n\in \{11,14,15\}$. As a result we conclude $\mc{K}_X(D)=\mathbb C$ and $\dim S_{\e^3,3/2}(n)=\dim \mc{K}_X(D)=1$. 

It remains to show that $\dim S_{\e^3,3/2}(23)=1$. The group $\G_0(23)$ has two cusps, one of width $23$ represented by $1$, and one with width $1$ represented by $1/23$. As a result we have $D=2x$ where $x$ again denotes the image of $1\in\hat{\QQ}$ under  $\HH\cup\hat{\QQ} \to X_0(23)$. Consider first $D'=x$. Repeating the above argument and using the fact that $X_0(23)$ has genus $2$ we conclude that $\dim \O_X(D')= \dim\mc{K}_X(D')=1$. From the fact that $\O_X(D') \supset \O_X(D)$ and the one dimensional space $\O_X(D')$ is given by the weight 2 cusp form $\eta(\t)^2\eta(23\t)^2$, we conclude that $\dim\O_X(D)=0$ and hence $\dim S_{\e^3,3/2}(23)=1$.
\end{proof}
Taking the above lemmas together we obtain our main result.
\begin{thm}
We have $\dim S_{\e^3,3/2}(n)=1$ if and only if $n$ coincides with the order of an element of the sporadic group $M_{23}$; viz., $n\in\{1,2,3,4,5,6,7,8,11,14,15,23\}$. 
\end{thm}

\section{Perspectives}\label{sec:persp}

As has been explained in the introduction, the fact that the space $S_{\e^3,3/2}(n)$ has dimension $1$ whenever $n$ is the order of an element of $M_{23}$ may, in light of the Rademacher sum construction of the McKay--Thompson series of Mathieu moonshine, be regarded as a counterpart to the genus zero property of monstrous moonshine since the latter is equivalent to the existence of Rademacher sum expressions for the monstrous McKay--Thompson series. This equivalence was extended in \cite{DunFre_RSMG} so as to obtain a Rademacher sum-based characterisation of the groups $\G_g<\SL_2(\RR)$ for $g$ in the monster by reformulating an earlier characterisation due to Conway--McKay--Sebbar \cite{ConMcKSebDiscGpsM}. The main result of this article, determining those $n$ for which $\dim S_{\e^3,3/2}(n)=1$, serves as evidence in support of an analogous characterisation of the discrete groups (and multipliers) of Mathieu moonshine. 

As in \S\ref{sec:intro} we attach a discrete group $\G_g=\G_0(n)<\SL_2(\ZZ)$ and a multiplier system $\psi_g=\rho_{n|h}\e^{-3}$ (cf. (\ref{eqn:rho})) with weight $1/2$ to each element $g\in M_{24}$. Following the suggestion made in \cite{Cheng:2012qc} we conjecture that the pairs $(\G,\psi)$ arising as $(\G_g,\psi_g)$ for $g\in M_{24}$ are exactly those of the form $(\G,\psi)=(\G_0(n),\rho_{n|h}\e^{-3})$ where 
\begin{enumerate}
\item
$n$ and $h$ are positive integers such that $h$ divides $\gcd(n,12)$, and 
\item
the Rademacher sum $R^{[-1/8]}_{\G,\psi,1/2}$ has shadow proportional to $\eta^3$.
\end{enumerate}
The main result of this article implies that if $h=1$ then the second condition above is necessarily satisfied when $n$ is the order of an element of $M_{23}$. It also suggests that the second condition is unlikely to be satisfied when $h=1$ and $n$ is not the order of an element of $M_{23}$, for then, according to our result, the space of possible shadow functions has dimension greater than one. 

Observe that the second condition actually implies that $R^{[-1/8]}_{\G,\psi,1/2}$ has vanishing shadow---and is thus a weak modular form---when $h>1$, for the multiplier of $\eta^3$ agrees with $\overline{\rho_{n|h}\e^{-3}}$ on $\G_0(nh)$ but differs from it on non-trivial cosets of $\G_0(nh)$ in $\G_0(n)$. A next step towards establishing the above conjectural characterization of the discrete groups and multipliers of Mathieu moonshine would be to carry out the $h>1$ analogue of the analysis presented in this article, showing that the space $S_{\rho_{n|h}\e^{-3},3/2}(n)$ of cusp forms for $\G_0(n)$ with weight $3/2$ and multiplier $\rho_{n|h}\e^{-3}$ vanishes if and only if $(n,h)=(n_g,h_g)$ for some fixed-point-free $g\in M_{24}$. The fact that $S_{\rho_{n|h}\e^{-3},3/2}(n)$ vanishes in case $(n,h)$ does arise from a fixed-point-free element of $M_{24}$ is established in \cite{Cheng2011}.

Another natural direction to explore is the generalisation of our results to umbral moonshine \cite{UM}. For $\ll$ an integer greater than $1$ and $0<r<\ll$ define
\begin{gather}
	S^{(\ll)}_r=\sum_{m\in\ZZ}(2\ll m+r)q^{(2\ll m+r)^2/4\ll}.
\end{gather}
Then $S^{(2)}_1$ recovers $\eta^3$ (cf. (\ref{eqn:intro:utheta})) and the vector-valued function $S^{(\ll)}=(S^{(\ll)}_1,\ldots,S^{(\ll)}_{\ll-1})$ is a vector-valued cusp form for the modular group with weight $3/2$ and a certain matrix-valued multiplier $\sigma^{(\ll)}$.
In \cite{UM} a weak vector-valued mock modular form $H^{(\ll)}$ is constructed for each $\ll\in\{2,3,4,5,7,13\}$ whose shadow is proportional to $S^{(\ll)}$, and evidence is collected in support of a family of conjectures relating $H^{(\ll)}$ to a finite group $G^{(\ll)}$ in analogy to the relationship between $H$ and $M_{24}$ described above. In fact Mathieu moonshine may be regarded as the special case of umbral moonshine that $\ll=2$; we have $H^{(2)}=H$ and $G^{(2)}=M_{24}$. For a positive integer $n$ define $S_{\s^{(\ll)},3/2}(n)$ to be the space of $(\ll-1)$-vector-valued cusp forms of weight $3/2$ for $\G_0(n)$ with multiplier system $\s^{(\ll)}$. Then $\dim S_{\s^{(\ll)},3/2}(n)\geq 1$ for all $n\geq 1$ since $S^{(\ll)}\in S_{\s^{(\ll)},3/2}(n)$. In light of the results of this paper, and upon inspection of the shadows described in \cite{UM}, it is tempting to suggest that the positive integers $n$ for which $\dim S_{\s^{(\ll)},3/2}(n)=1$ are exactly those such that there exists an element $g \in G^{(\ll)}$ of order $n$ with $\chi^{(\ll)}_g = \overline\chi^{(\ll)}_g >0 $. This condition means that $g$ has fixed points but no anti-fixed points in the natural signed permutation representation of $G^{(\ll)}$.  We refer to \cite{UM} for more details.

As is still the case for Ogg's observation on primes dividing the order of the monster, a conceptual explanation of the results of this article is not yet available. 
Nonetheless,  we have seen in \S\ref{sec:intro} that both results may be motivated by the Rademacher sum construction of the Mathieu and monstrous McKay--Thompson series. 
As a result, it is tempting to envisage a class of algebraic structures whose twisted characters are given by Rademacher sums. 
The relation between this algebraic structure and vertex algebra should be a chiral counterpart to the relation between string theory on $AdS_3$ geometries and conformal field theory in two dimensions. See \cite{DunFre_RSMG,Cheng2011} for a more detailed discussion.

In the monstrous case this is a central implication of the results of \cite{DunFre_RSMG}, and the construction of a such an object with the monster acting as automorphisms would explain the genus zero property. This is because, as mentioned before, the relevant Rademacher sums have the required invariance if and only if the underlying discrete groups have genus zero. The existence of elements of prime order in the monster with associated discrete group of the form $\G_0(p)^+$ then explains Ogg's observation. For the case of Mathieu moonshine the untwisted character is a Rademacher sum of weight $1/2$ with multiplier system $\psi=\e^{-3}$ and we can expect 
the twisted characters to have shadows proportional to that of the untwisted character, which is to say, proportional to $\eta^3$ since $\dim S_{\e^3,3/2}(1)=1$ according to our results.  As such, the construction of this  algebraic structure would furnish an algebraic foundation for the conjectural characterisation of the discrete groups of Mathieu moonshine described above, and it would explain our analogue of Ogg's result: that the primes dividing the order of $M_{24}$ are exactly those for which $\dim S_{\e^3,3/2}(p)=1$.

\appendix

\section{Special Functions}\label{sec:fun}

For the exponential $x\mapsto x^s$ with $s$ a non-integer we employ the principal branch of the logarithm, so that 
\begin{gather}\label{eqn:fun:pbranch}
	x^s=|x|^se^{\ii s\th}
\end{gather}
when $x=|x|e^{\ii\th}$ for $-\pi<\th\leq \pi$.

The {\em Dedekind eta function}, denoted $\eta(\t)$, is a holomorphic function on the upper half-plane defined by the infinite product 
\begin{gather}\label{eqn:fun:Ded}
	\eta(\t)=q^{1/24}\prod_{n>0}(1-q^n)
\end{gather}
where $q=\ex(\t)=e^{\tpi \t}$. It is a modular form of weight $1/2$ for the modular group $\SL_2(\ZZ)$ with multiplier $\e:\SL_2(\ZZ)\to\CC$ so that
\begin{gather}\label{eqn:fun:eps}
	\eta(\g\t)\e(\g)\jac(\g,\t)^{1/2}=\eta(\t)
\end{gather}
for all $\g = \left(\begin{smallmatrix} a&b\\ c&d \end{smallmatrix}\right) \in\SL_2(\ZZ)$, where $\jac(\g,\t)=(c\t+d)^{-1}$. The {\em multiplier system} $\e$ may be described explicitly as 
\begin{gather}\label{eqn:fun:dedmlt}
\e\bem a&b\\ c&d\eem 
	=
\begin{cases}
	\ex(-b/24),&c=0,\,d=1\\
	\ex(-(a+d)/24c+s(d,c)/2+1/8),&c>0
\end{cases}
\end{gather}
where $s(d,c)=\sum_{m=1}^{c-1}(d/c)((md/c))$ and $((x))$ is $0$ for $x\in\ZZ$ and $x-\lfloor x\rfloor-1/2$ otherwise. We can deduce the values $\e(a,b,c,d)$ for $c<0$, or for $c=0$ and $d=-1$, by observing that $\e(-\g)=\e(\g)\ex(1/4)$ for $\g\in\SL_2(\ZZ)$. 

For an integer $a$ and an integer $n=\varepsilon \prod_{i=1}^k p_i^{e_i}$ with $\varepsilon=\pm1$ and the $p_i$ mutually distinct primes, the {\em Kronecker symbol} $\left(\frac{a}{n}\right)$ is defined as follows.
First we have the multiplication rule
$$
\left(\frac{a}{n}\right) = \left(\frac{a}{\varepsilon}\right)\prod_{i=1}^k  \left(\frac{a}{p_i}\right)^{e_i}
$$
where the first factor is given by 
\begin{gather}
\left(\frac{a}{1}\right)  = 1,\qquad
\left(\frac{a}{-1}\right) =\begin{cases} -1 & a<0, \\ 1 & a\geq 0, \end{cases}\;
\end{gather}
and for odd primes $p$ the Kronecker symbol $\left(\frac{p}{n}\right)$ is identical to the Legendre symbol 
\begin{align}
\left(\frac{a}{p}\right) = \begin{cases} 0 & p\lvert a \\ 1 & a {\text{ is a quadratic residue mod } } p \\  -1 & a {\text{ is not a quadratic residue mod }} p\end{cases}
\end{align}
and coincides with $\left(\frac{a}{p}\right)  = a^{(p-1)/2}$ mod $p$. 
Finally, for $p=2$ we have 
\begin{align}
\left(\frac{a}{2}\right) = \begin{cases} 0 & 2\lvert a, \\ 1 & a= \pm 1 {\text{ mod }} 8,\\  -1 & a= \pm 3 {\text{ mod }} 8. \end{cases}
\end{align}


\addcontentsline{toc}{section}{References}
\bibliographystyle{alpha}
\bibliography{gpsm24_tex}
\end{document}